\documentclass{amsart}
\usepackage{bm}
\usepackage{ifthen}
\usepackage{blkarray}
\usepackage{makecell}
\usepackage{stmaryrd}
\usepackage{color}
\usepackage{longtable}
\usepackage{amsrefs}
\usepackage{youngtab}

\theoremstyle{plain}
\newtheorem{thm}{Theorem}

\newtheorem{cor}[thm]{Corollary}

\theoremstyle{definition}
\newtheorem{dfn}[thm]{Definition}
\newtheorem{exam}[thm]{Example}

\theoremstyle{remark}

\newtheorem*{rem*}{Remark}
\numberwithin{equation}{section}

\newcommand{\C}{\mathbb C}
\newcommand{\R}{\mathbb R}
\renewcommand{\H}{\mathbb H}
\newcommand{\GL}{{\rm GL}}
\newcommand{\gl}{\mathfrak{gl}}
\newcommand{\F}[2][]{
   \ifthenelse{ \equal{#1}{} }
      {F^{#2}_n}
      {F^{(#1,#2)}_n}
}
\newcommand{\gF}[2]{F^{#1}_{#2}}
\newcommand{\tF}[2]{
   \widetilde F^{#1}_{#2}}
\renewcommand{\c}[2]{c^{#1}_{#2}}
\newcommand{\LRC}[2]{{\rm LRC}\:^{#1}_{#2}}
\newcommand{\bmu}{\bm{\mu}}
\newcommand{\bla}{\bm{\lambda}}
\newcommand{\blacol}[1]{\bla_{\bullet#1}}
\newcommand{\blarow}[1]{\bla_{#1\bullet}}
\newcommand{\la}{\lambda}
\renewcommand{\hom}{{\rm Hom}}
\newcommand{\ch}{{\rm ch}}
\newcommand{\U}{{\rm U}}
\renewcommand{\O}{{\rm O}}
\newcommand{\Sp}{{\rm Sp}}
\renewcommand{\sp}{\mathfrak{sp}}
\newcommand{\tE}[2]{\widetilde E^{#1}_{#2}}
\renewcommand{\det}{{\rm det}}

\begin{document}

\title[Contingency tables and Littlewood--Richardson coefficients]{Contingency tables and the generalized Littlewood--Richardson coefficients}

\author{Mark Colarusso}
\address{
Mark Colarusso \\
University of South Alabama \\ Department of Mathematics and Statistics \\
Mobile, AL 36608} 
\email{mcolarusso@southalabama.edu}

\author{William Q. Erickson}
\address{
William Q. Erickson\\
University of Wisconsin--Milwaukee \\
Department of Mathematical Sciences \\
P.O. Box 0413 \\
Milwaukee WI 53201}
\email{wqe@uwm.edu}

\author{Jeb F. Willenbring}
\address{
Jeb F. Willenbring \\
University of Wisconsin--Milwaukee \\
Department of Mathematical Sciences \\
P.O. Box 0413 \\
Milwaukee WI 53201} \email{jw@uwm.edu}

\begin{abstract}
The Littlewood--Richardson coefficients $c^\la_{\mu\nu}$ give the multiplicity of an irreducible polynomial $\GL_n$-representation $\F{\la}$ in the tensor product of polynomial representations $\F{\mu}\otimes \F{\nu}$.  In this paper, we generalize these coefficients to an $r$-fold tensor product of \emph{rational} representations, and give a new method for computing them using an analogue of statistical contingency tables.  We demonstrate special cases in which our method reduces to counting statistical contingency tables with prescribed margins.  Finally, we extend our result from the general linear group to both the orthogonal and symplectic groups.
\end{abstract}





\maketitle

\section{Introduction}

The finite-dimensional real associative division algebras are the real numbers $\R$, the complex numbers $\C$, and the quaternions $\H$.  Each of these has a norm defined by conjugation.  Let $A= \R,\C,$ or $\H$, and consider the group of invertible $A$-linear transformations of $A^n$ preserving this norm; this group is a real compact Lie group, and is denoted $\O(n)$, $\U(n)$, or $\Sp(n)$, respectively.  Complexifying these groups, we obtain the complex classical groups $\O(n,\C)$, $\GL(n,\C)$, and $\Sp(2n,\C)$.  Our goal in this paper is to develop a simple combinatorial description of tensor product multiplicities for rational representations of these classical groups.  We will do this by studying the complexifications of the \emph{compact dual pairs} (see \cite{howeremarks} and \cite{howetrans}) associated with the compact real forms above.  We devote the majority of the paper to proving the $\GL(n,\C)$ case in detail.  In the final section, we make the necessary changes to extend the result to $\O(n,\C)$ and $\Sp(2n,\C)$.

In more detail, let $GL_n =\GL(n,\C)$ be the complex general linear group of $n \times n$ invertible complex matrices, which  has the structure of an affine variety.  Its coordinate ring $\C[\GL_n]$ consists of rational functions of the form $\frac{N}{D}$, where $N$ is a polynomial in the matrix entries $x_{ij}$, and $D$ is a nonnegative power of the determinant.  If $G \subseteq \GL_n$ is an algebraic subgroup, then a \textit{rational representation} of $G$ is a group homomorphism $\rho: G \rightarrow \GL_m$ for some $m$, which is a morphism in the category of affine varieties.  In the case that $G = \GL_n$, we will say that a rational representation $\rho : \GL_n \rightarrow \GL_m$ is \emph{polynomial} if the matrix coordinates $\rho(g)_{i,j}$ are polynomials in the entries of $g \in \GL_n$. 

In the existing literature, tensor product multiplicities for rational representations of $\GL_n$ are computed by tensoring with a sufficiently high power of the determinant in order to yield polynomial representations.  Keeping track of these powers of the determinant amounts to bookkeeping that is not always conceptually transparent.  In this paper, we obtain the desired tensor product multiplicities as a special case of Howe duality and seesaw reciprocity, and we express them in terms of natural generalizations of contingency tables from statistics which we refer to as Littlewood--Richardson (or simply ``LR") contingency tables. In fact, in certain special cases these LR-contingency tables specialize to classical contingency tables, and the corresponding tensor product multiplicities become the number of contingency tables with certain fixed margins. Computing the number of such tables is an important open problem in statistics, and we find it interesting that the answer to this question can be expressed in terms of certain classical objects in the representation theory of $\GL_n$.  Our approach can also be used to express tensor product multiplicities for representations of the orthogonal and symplectic groups in terms of \emph{symmetric} LR-contingency tables.  In this manner, we develop a unified framework using Howe duality to compute tensor product multiplicities for any complex classical group.

\section{Preliminaries}
\label{s:prelim}

We begin by recalling the standard labeling of rational representations of $\GL_n$.

\subsection*{Labeling representations of $\GL_n$} A \emph{partition} is a finite, weakly decreasing sequence of positive integers (\emph{parts}), typically denoted by a lowercase Greek letter.  When we need to write out a partition $\lambda$ explicitly, we will write the $i^{\rm th}$ part as $\lambda(i)$, rather than the more standard $\lambda_i$, since we will use subscripts later to index a vector of partitions.  The \emph{length} of a partition $\la=(\la(1),\dots, \la(s))$ will be denoted by $\ell(\la) := s$, while the \emph{size} is defined as $|\la| := \sum_{i=1}^s \la(i)$.
The \emph{empty partition} will be denoted by $0$ and has length and size zero.  Given a partition $\la$, the conjugate partition $\la^\prime$ is the partition whose $j^{\text{th}}$ part is the number of elements in the set $\{i:\la(i) \geq j \}$; alternatively, if we identify a partition with its Young diagram, then $\la'$ is the transpose of $\la$.  Given a partition $\la$ with $\ell(\la)=s$ and a positive integer $n \geq s$, we may identify $\la$ with the $n$-tuple of nonnegative integers $(\la(1), \dots, \la(s), 0, \dots, 0)$, by ``padding out'' with zeros on the right-hand end.

By the theorem of the highest weight, each irreducible rational representation of $\GL_n$ may be labeled by an $n$-tuple of weakly decreasing integers. Specifically, let $\{e_1, \dots, e_n\}$ denote the standard basis of $\C^n$.  Then a weakly decreasing $n$-tuple $(a_1, \dots, a_n) \in \mathbb Z^n$ corresponds to the linear functional sending $e_i$ to $a_i$ for each $i$, which defines a dominant integral highest weight of $\GL_n$, with respect to the standard Borel subgroup of upper triangular matrices and algebraic torus $T$ of diagonal matrices.  The corresponding representation is polynomial if and only if each $a_i$ is nonnegative, i.e., the $n$-tuple is a partition; otherwise, the representation is rational but not polynomial.

In order to label rational representations, we will expand our use of lowercase Greek letters beyond partitions, to denote $n$-tuples of weakly decreasing integers, with the nonnegative condition dropped. It will also sometimes be useful to denote such $n$-tuples by ordered pairs of partitions, as follows.  Suppose we have two partitions $\alpha$ and $\beta$, where $\ell(\alpha)=a$ and $\ell(\beta)=b$, with $a+b\leq n$.  Then we can construct an $n$-tuple of weakly decreasing integers:
\begin{equation}
\label{combine}
    (\alpha,\beta):=
    \Big(\alpha(1),\dots,\alpha(a),0,\dots,0,-\beta(b),\dots,-\beta(1)\Big).
\end{equation}
Conversely, an $n$-tuple $\lambda$ of weakly decreasing integers may be expressed uniquely as the ordered pair of partitions $\la=(\la^+,\la^-)$, where $\la^+$ consists of the positive coordinates in $\lambda$, while $\la^-$ consists of the negative coordinates with signs changed and order reversed.  Any intervening 0's in $\lambda$ are ignored without ambiguity.  It is important to reiterate that both $\la^+$ and $\la^-$ are true partitions, whereas $\la$ is not.

We write $\F[\alpha]{\beta}$ to denote the irreducible rational representation of $\GL_n$ with highest weight $(\alpha,\beta)$ in the standard coordinates described above. We write $\F{\alpha}$ for $\F[\alpha]{0}$ in which case the representation is polynomial.  Note also that the dual representation $(\F[\alpha]{\beta})^*$ is equivalent to $\F[\beta]{\alpha}$.  Thus, $(\F{\alpha})^* \cong \F[0]{\alpha}$.  Finally, $\F{(1)}\cong \C^n$ is the defining representation of $\GL_n$, while $\F[(1)]{(1)}$ is the nontrivial component of the adjoint representation of $\GL_n$ on its Lie algebra $\gl_n$.

\subsection*{Littlewood--Richardson coefficients}
In general, if $V$ and $W$ are representations of a group $G$, then we will set $[V:W] := \dim \hom_G(V,W)$.  Note that $[V:W]$ may be an infinite cardinal; if $V$ and $W$ are finite-dimensional, however, then $[V:W]$ is a nonnegative integer.  Furthermore, if $V$ is irreducible and $W$ is completely reducible as a $G$-representation, then $[V:W]$ is the multiplicity of $V$ in the isotypic decomposition of $W$.

With this notation in mind, let $\la,\mu,\nu$ be partitions with length at most $n$.  For all $k \geq n$, it is a standard fact (see \cite{htw}) that $\left[\F{\la}: \F{\mu} \otimes \F{\nu} \right] =
    \left[F^{\la}_k: F^{\mu}_k \otimes F^{\nu}_k \right]$.  Thus, there is no ambiguity in defining $\c{\la}{\mu \nu} := [\F{\la}: \F{\mu} \otimes \F{\nu}]$.  The numbers $\c{\la}{\mu \nu}$ are commonly called the \emph{Littlewood--Richardson coefficients}, and have a long history in both combinatorics and representation theory.  (See \cite{fulton}, \cite{howe}, and \cite{macdonald}.)

\subsection*{Statement of the problem} In order to generalize these coefficients, we will use a boldface Greek letter $\bmu$ to denote a vector $(\mu_1,\dots,\mu_r)$ of weakly decreasing $n$-tuples:
\begin{equation}
\label{bmudef}
    \bmu=\bigg((\mu^+_1,\mu^-_1),\dots,(\mu^+_r,\mu^-_r)\bigg),
\end{equation}
where $\mu_i = (\mu_i^+,\mu_i^-)$ for a unique pair of partitions $\mu^+_i$ and $\mu^-_i$, as in \eqref{combine}.  It will be important to extract the positive (respectively, negative) components from each of the $\mu_i$ and combine them into a single vector, so we introduce the notation
\begin{equation}
\label{bmupm}
    \bmu^+:=(\mu^+_1,\dots,\mu^+_r), \qquad
    \bmu^-:=(\mu^-_1,\dots,\mu^-_r).
\end{equation}
Given a weakly decreasing $n$-tuple $\la$, and given $\bmu$ as in \eqref{bmudef}, we let
\begin{equation}
\label{eq:poly}
    \LRC{\la}{\bmu} := \left[ \F{\la}:
    \F{\mu_1} \otimes \cdots \otimes \F{\mu_r} \right]
\end{equation}
be the \emph{generalized Littlewood--Richardson coefficient}.  Note that the ``$c$" notation is used for \emph{polynomial} representations of $\GL_n$, whereas ``LRC" is for an arbitrary tensor product of \emph{rational} representations of $\GL_n$.

To compute \eqref{eq:poly}, it suffices to assume that $(\lambda^+,\lambda^-)=(0,0)=0$ is the empty partition.
Indeed,

\begin{align*}
\LRC{\la}{\bmu}&=\left[\F{\lambda}: \F{\mu_1}\otimes\dots\otimes \F{\mu_r}\right]\\
&= \dim \hom_{\GL_n}\left( \F{\la},\:\F{\mu_1}\otimes\dots\otimes \F{\mu_r} \right)\\
&=\dim \left( \left(\F{\la}\right)^* \otimes \F{\mu_1}\otimes\dots\otimes \F{\mu_r} \right)^{\GL_n}\\
&=\dim \hom_{\GL_n}\left(\F{0},\:\left(\F{\la}\right)^* \otimes \F{\mu_1}\otimes\dots\otimes \F{\mu_r} \right)\\
&=\left[ \F{0}:\F[\la^-]{\la^+}\otimes \F{\mu_1}\otimes \cdots \otimes \F{\mu_r} \right]\\
&=\LRC{0}{((\la^-,\la^+),\:\bmu). }
\end{align*}
In the case $r=2$, a formula for this generalized LRC was derived in \cite{king}, and the same result was proved in \cite{htw} using the notion of Howe duality.  As our main result in this paper, we will compute $\LRC{0}{\bmu}$, for arbitrary $r$, using an analogue of statistical contingency tables.

\subsection*{General Weyl duality theorem}

In the proof of our main theorem in Section \ref{results}, we will encounter a specific instance of \textit{Howe duality} (see \cite{htw}).  We begin by recalling a general duality result (Theorem 5.6.1 in \cite{gw}).  

Let $V$ be a finite-dimensional complex vector space, let $\C[V]$ be the algebra of polynomial functions on $V$, and let $\mathbb D(V)$ be the Weyl algebra of polynomial-coefficient differential operators on $V$.  Suppose $G \subseteq \GL(V)$ is a reductive algebraic subgroup.  The natural action of $G$ on $V$ induces actions of $G$ on $\C[V]$ and $\mathbb D(V)$.  Let $\mathcal A:=\mathcal A[G]$ denote the group algebra of $G$, and let $\mathcal B:=\mathbb D(V)^G \subseteq \mathbb D(V)$ be the $G$-invariant elements of the Weyl algebra.  Note that both $\mathcal A$ and $\mathcal B$ are semisimple associative algebras.  There is a canonical multiplicity-free decomposition
\begin{equation}
\label{dualdecomp}
    \C[V]\cong \bigoplus_{\la} F^\la \otimes \widetilde F^\la
\end{equation}
as an $\mathcal A \otimes \mathcal B$-module, where $\la$ is a parameter indexing equivalence classes of irreducible rational representations of $G$.  (The particular irreducible representations that appear in \eqref{dualdecomp} depend on the specific group $G$ in question; we will describe these representations for $G=\GL_n$ in the next subsection.)  In each summand of \eqref{dualdecomp}, the first tensor factor $F^\la$ is the irreducible representation of $G$ labeled by $\la$, and $\widetilde F^\la$ is an irreducible $\mathcal B$-module that uniquely determines $\la$.  In general, $\mathcal B$ is generated as an associative algebra by a finite set of operators that span a Lie algebra $\mathfrak g'$, and $\widetilde F^\la$ is also an irreducible $\mathfrak g'$-module.  We then have a \textit{dual reductive pair} given by $G$---$\mathfrak g'$.

\subsection*{Special case: $\GL_n$---$\gl_{p+q}$ Howe duality}
To prove our main result, we use the special case $G=\GL_n$ and $\mathfrak g'=\gl_{p+q}$, for positive integers $n,p,q$.  In the setting of \eqref{dualdecomp}, set 
\begin{equation*}
    V=M_{n,p}\oplus M_{n,q},
\end{equation*}
where $M_{n,p}$ (respectively $M_{n,q}$) denotes the space of $n\times p$ (respectively $n\times q$) complex matrices.  The action of $\GL_n$ on $V$ via $g\cdot(X,Y)=(gX,(g^{-1})^tY)$ induces an action on $\C[V]$ in the usual way.

From the original perspective of Roger Howe's work (see \cite{howeremarks} and \cite{howetrans}), this dual reductive pair $\GL_n$---$\gl_{p+q}$ is the complexification of the \textit{compact dual pair} $\U(n)$---$\mathfrak{su}(p,q)$.  Indeed, $G=\GL_n$ is the complexification of the compact real group $G_0=\U(n)$, corresponding to the real division algebra $\C$.  (In our final section, we will consider the cases where $G_0$ is $\O(n)$ and $\Sp(n)$, corresponding to $\R$ and $\H$ respectively.)  Meanwhile, $\mathfrak g'=\gl_{p+q}$ is the complexification of the real Lie algebra $\mathfrak g'_0=\mathfrak u(p,q)$.\footnote{It is truer to the spirit of Howe's work to consider $\mathfrak{su}(p,q)$ rather than $\mathfrak u(p,q)$, since the former is a \emph{Hermitian symmetric} real Lie algebra.  Since the center of $\mathfrak u(p,q)$ acts trivially, however, there is no harm in extending to a $\mathfrak{u}(p,q)$-action, so that the complexification is $\gl_{p+q}$.} 

Note that $\mathfrak g'_0$ has maximal compact subalgebra $\mathfrak k'_0=\mathfrak u(p)\oplus \mathfrak u(q)$, embedded block-diagonally; complexifying $\mathfrak k'_0$, we obtain the subalgebra $\mathfrak k' = \gl_p \oplus \gl_q$ of block diagonal matrices in $\mathfrak g'$.  Key to the proof of our main result will be the restriction of the $\mathfrak g'$-action to $\mathfrak k'$. As a consequence of the compactness of $G_0$, the $\mathfrak g'$-modules $\widetilde F^\la$ are completely reducible with finite multiplicities when restricted to $\mathfrak k'$.

We now describe the $\gl_{p+q}$-action on $\C[V]$.  Let $x_{ij}$ and $y_{ij}$ denote the coordinate functions on a pair of matrices $(X,Y)\in V$, and consider a $(p+q)\times (p+q)$ matrix $D$ of differential operators
\begin{equation*}
    D=
    \left[
    {\renewcommand{\arraystretch}{1.5}
    \begin{array}{c|c}
    -\mathfrak k_x(i,j) & \Delta^2(i,j) 
    \\ \hline
    -h^2(i,j) & \mathfrak k_y(i,j)
    \end{array}}
\right]_{\textstyle ,}
\end{equation*}
written in $2\times 2$ block form, where the indexing begins at $(1,1)$ within each individual block, and where

\begin{itemize}
\setlength\itemsep{5pt}
\item $\mathfrak k_x(i,j)=\sum_{k=1}^n x_{kj} \frac{\partial}{\partial x_{ki}}+n\delta_{ij}$ (shifted Euler operators), for $1\leq i,j \leq p$;
\item $\mathfrak k_y(i,j)=\sum_{k=1}^n y_{ki} \frac{\partial}{\partial y_{kj}}$ (Euler operators), for $1 \leq i,j \leq q$;
\item $\Delta^2(i,j)=\sum_{k=1}^n \frac{\partial^2}{\partial x_{ki} \partial y_{kj}}$ (raising operators), for $1\leq i \leq p$ and $1 \leq j \leq q$;
\item $h^2(i,j)=\sum_{k=1}^n x_{kj}y_{ki}$ (lowering operators), for $1\leq i \leq q$ and $1 \leq j \leq p$.
\end{itemize}
(See \cite{htw} and the references therein.)  The $(p+q)^2$ operators in $D$ generate $\mathbb D(V)^G$, and they form a basis for a Lie algebra isomorphic to $\gl_{p+q}$ as follows.  Let $e_{k\ell}\in \gl_{p+q}$ denote the matrix with 1 in the $(k,\ell)$ position and 0's elsewhere. Then the map $e_{k\ell}\mapsto D_{k\ell}$ extends to a homomorphism of Lie algebras $\gl_{p+q}\rightarrow \mathbb D(V)^G$, and $\gl_{p+q}$ acts on $\C[V]$ via its image in $\mathbb D(V)^G$.

For the pair $\GL_n$---$\gl_{p+q}$, the dual decomposition from \eqref{dualdecomp} has the realization
\begin{equation}
\label{typeadualdecomp}
   \C[V]\cong \bigoplus_\la \F{\la} \otimes \tF{\la}{p,q},
\end{equation}
where $\la=(\la^+,\la^-)$ ranges over all weakly decreasing integer $n$-tuples such that $\ell(\la^+)\leq p$ and $\ell(\la^-)\leq q$.   For each such $\la$, the second tensor factor $\tF{\la}{p,q}$ is an irreducible, infinite-dimensional, highest-weight $\gl_{p+q}$-module.
Under the assumption $n\geq p+q$, known as the \textit{stable range}, $\tF{\la}{p,q}$ has additional structure under the restriction to the action of $\mathfrak k'=\gl_p \oplus \gl_q$:
\begin{equation}
\label{eq:kmodule}
    \tF{\la}{p,q}\cong\C\left[M_{p,q}\right]\otimes \left(\gF{\la^+}{p}\otimes \gF{\la^-}{q}\right)_{\textstyle .}
\end{equation}
Since $\mathfrak k'$ is spanned by Euler operators, which preserve the degree of homogeneous polynomial functions, the $\mathfrak k'$-action on \eqref{eq:kmodule} is locally finite. Therefore the action integrates to an action by the group $K'=\GL_p \times \GL_q$, which up to a central shift (introduced by the $\delta_{ij}$ in the Euler operators) is given as follows. The group $K'$ acts on the right-hand side of \eqref{eq:kmodule} with the obvious action on $\gF{\la^+}{p}\otimes \gF{\la^-}{q}$, while the action on $\C[M_{p,q}]$ is given by \begin{equation}
\label{eq:kaction}
((g,h)\cdot f)(X)=f(g^t X h)
\end{equation}
for $g \in \GL_p$, $h\in \GL_q$, $f \in \C[M_{p,q}]$, and $X\in M_{p,q}$.  Under the action in \eqref{eq:kaction},
\begin{equation}
    \label{howepq}
    \C[M_{p,q}]\cong \bigoplus_\la \gF{\la}{p} \otimes \gF{\la}{q}
\end{equation} as a $K'$-representation, where the direct sum is over all partitions $\la$ such that $\ell(\la)\leq \min\{p,q\}$.  (See Corollary 5.6.8 in \cite{gw}.)

\subsection*{Seesaw pairs} Returning to the general setting, consider a dual pair $G$---$\mathfrak g'$ acting on $\C[V]$.  Restricting to a reductive subgroup $H\subset G$ induces another dual pair $H$---$\mathfrak h'$, with corresponding decomposition
\begin{equation*}
    \C[V]\cong\bigoplus_{\la} E^\la \otimes \widetilde E^\la.
    \end{equation*}
The key fact is that $\mathfrak g' \subseteq \mathfrak h'$, which is not obvious but can be verified on a case-by-case basis as in \cite{howeremarks}.  Then we say that $G$---$\mathfrak g'$ and $H$---$\mathfrak h'$ are a \emph{seesaw pair} on the space $\C[V]$, which we display as
\begin{equation*}
\begin{array}{ccc}
    G & \longleftrightarrow & \mathfrak g'  \\
    \uparrow & & \downarrow \\
    H & \longleftrightarrow & \mathfrak h'.
\end{array}
\end{equation*}
The term ``seesaw" evokes the reciprocity of multiplicities
\begin{equation}
\label{eq:seesaw}
     \left[ E^\la : F^\mu \right] = \left[ \widetilde F^\mu : \widetilde E^\la \right].
\end{equation}
This equality, while not \emph{a priori} clear, is a consequence of the $G$- and $H$-actions being the complexifications of actions by compact groups.  In particular, in our setting, $H$ has compact form $H\cap G_0$.  It follows that both actions on the left-hand side of the seesaw pair are completely reducible, so the actions of the Lie algebras on the right-hand side are completely reducible as well. (See \cite{howetrans} and \cite{hk}).  

In this paper, we will consider the seesaw pair
\begin{equation*}
\begin{array}{ccc}
\overbrace{\GL_n\times \cdots \times \GL_n}^r & \longleftrightarrow & \gl_{p_1+q_1}\oplus \cdots \oplus \gl_{p_r+q_r}\\
\uparrow & & \downarrow\\
\Delta\left(\GL_n\right) &\longleftrightarrow & \gl_{P+Q,}\end{array}
\end{equation*}
where $P=\sum_i p_i$ and $Q=\sum_i q_i$.  The symbol $\Delta(\GL_n)$ denotes the diagonal embedding of $\GL_n$ given by $g \mapsto (g,\dots,g)$, while the embedding on the right-hand side is block-diagonal.  

Note that the dual pair on the bottom is the same as we discussed in the previous subsection, with $P$ and $Q$ playing the roles of $p$ and $q$; hence the stable-range condition becomes $n \geq P+Q$.  It follows that 
\begin{align*}
\C[V]&=\C\left[M_{n,P}\oplus M_{n,Q}\right]\\
&\cong \C\left[\left( M_{n,p_1} \oplus M_{n,q_1}\right) \oplus \cdots \oplus \left( M_{n,p_r} \oplus M_{n,q_r}\right)\right]\\
&\cong \C[M_{n,p_1}\oplus M_{n,q_1}] \otimes \cdots \otimes \C[M_{n,p_r} \oplus M_{n,q_r}],
\end{align*}
on which the top dual pair $\GL_n\times \cdots \times \GL_n$---$\gl_{p_1+q_1}\oplus \cdots \oplus \gl_{p_r+q_r}$ acts naturally, as described in the previous subsection.  Now the seesaw reciprocity in \eqref{eq:seesaw} implies
\begin{equation}
\label{eq:lrcseesaw}
    \LRC{\la}{\bmu} = \left[ \tF{\mu_1}{p_1,q_1} \otimes \cdots \otimes \tF{\mu_r}{p_r,q_r} : \tF{\la}{P,Q}\right].
\end{equation}
In the next section, we will apply \eqref{eq:lrcseesaw} when $\la=0$ to compute $\LRC{0}{\bmu}$.

\section{Main result}
\label{results}

In statistics, a \textit{contingency table} is a matrix of nonnegative integers with prescribed row and column sums.  These row and column sums are referred to as the \textit{margins} of the table.  A matrix whose diagonal entries are 0 is said to be \textit{hollow} (see \cite{gentle}, p.\ 42).  Our main result will use an analogue of a hollow contingency table, in which the nonnegative integers are replaced by partitions (labeling polynomial representations of $\GL_n$), and addition along the rows and columns is replaced by taking generalized LRC's.

Let $\bmu$ be as in \eqref{bmudef}, with the vectors $\bmu^+$ and $\bmu^-$ as in \eqref{bmupm}.  Then imagine $\bmu^+$ and $\bmu^-$ arranged outside an $r \times r$ matrix of partitions $\Lambda = (\la_{ij})$ as follows:

\begin{equation*}
\begin{blockarray}{l c c c c}
& \mu^-_1 & \mu^-_2 & \cdots & \mu^-_r\\[2ex]
\begin{block}{l [c c c c]}
\mu^+_1 \phantom{xx}& \la_{1,1} & \la_{1,2} & \dots &\la_{1,r}\\[1.5ex]
\mu^+_2 &\la_{2,1} & \la_{2,2} & \cdots & \la_{2,r}\\
\:\vdots &\vdots & \vdots &    \ddots & \vdots\\
\mu^+_r &\la_{r,1} & \la_{r,2} & \cdots & \la_{r,r}\\ \end{block}
\end{blockarray}
\end{equation*}
Let $\blarow{i}:=(\la_{i,1},\ldots,\la_{i,r})$ denote the vector of partitions in the $i^\text{th}$ row of $\Lambda$, and $\blacol{j}:=(\la_{1,j},\ldots,\la_{r,j})$ the vector of partitions in the $j^\text{th}$ column of $\Lambda$.  Consider $\LRC{\mu^+_i}{\blarow{i}}$ and $\LRC{\mu^-_j}{\blacol{j}}$ for $1 \leq i,j \leq r$.  We write $\|\cdot \|_{\bmu}$ to denote the product of all these generalized LRC's.

\begin{dfn}
Let $\Lambda$ be an $r\times r$ matrix of partitions $\la_{ij}$, with $\bmu$ as in \eqref{bmudef}.  Then
\begin{equation*}
    \|\Lambda\|_{\bmu}:=\prod_{i=1}^r \left(\LRC{\mu^+_i}{\blarow{i}} \right) \left(\LRC{\mu^-_i}{\blacol{i}}\right)_{\textstyle .}
\end{equation*}
\end{dfn}

Now we define our analogue of a statistical contingency table, where the ``LR" stands for ``Littlewood--Richardson."
\begin{dfn}
\label{dfn:LRCT}
Let $\bmu$ be as in \eqref{bmudef}.  An \emph{LR-contingency table with margins $\bmu$} is an $r \times r$ matrix of partitions $\Lambda=(\la_{ij})$ such that $\| \Lambda \|_{\bmu} \neq 0$.  If $\la_{ii}=0$ for all $i$, then we say that $\Lambda$ is \textit{hollow}.
\end{dfn}
Readers familiar with the Littlewood--Richardson rule will see that in an LR-contingency table $\Lambda$ with margins $\bmu$, we must have $\ell(\la_{ij})\leq \min\{\ell(\mu^+_i),\ell(\mu^-_j)\}$.

Our main result is the following expression for the generalized LRC in \eqref{eq:poly} with $\la=0$, in terms of hollow LR-contingency tables.

\begin{thm}\label{thm:bigthm}
Let $\bmu$ be as in \eqref{bmudef}.  If $n \geq \sum_{i=1}^r \ell(\mu^+_i)+\ell(\mu^-_i)$, then
\begin{equation*}
\LRC{0}{\bmu}=\sum_\Lambda \| \Lambda \|_{\bmu,}
\end{equation*}
where $\Lambda$ ranges over all hollow LR-contingency tables with margins $\bmu$.
\end{thm}

\begin{proof}
We consider the seesaw pair discussed in
Section \ref{s:prelim}:
\begin{equation*}
\begin{array}{ccc}
\overbrace{\GL_n\times \cdots \times \GL_n}^r & \longleftrightarrow & \gl_{p_1+q_1}\oplus \cdots \oplus \gl_{p_r+q_r}\\
\uparrow & & \downarrow\\
\Delta\left(\GL_n\right) &\longleftrightarrow & \gl_{P+Q}\end{array}
\end{equation*}
We choose $p_i=\ell(\mu^+_i)$, and $q_i=\ell(\mu^-_i)$.  Since $n \geq \sum_{i=1}^r \ell(\mu^+_i)+\ell(\mu^-_i)$, we are in the stable range.

Let
 \begin{equation}
 \label{eq:reciprocity}
m_{\bmu}:=\left[\tF{\mu_1}{p_1,q_1} \otimes \cdots \otimes \tF{\mu_r}{p_r,q_r} : \tF{0}{P,Q}\right]
\end{equation}
denote the multiplicity of the $\bigoplus_i \gl_{p_i+q_i}$-module $\tF{\mu_1}{p_1,q_1} \otimes \cdots \otimes \tF{\mu_r}{p_r,q_r}$ in $\tF{0}{P,Q}$.  By the seesaw reciprocity from \eqref{eq:lrcseesaw}, we have
\begin{equation}
\label{mmu}
\LRC{0}{\bmu}=m_{\bmu}.
\end{equation}
To compute the multiplicity $m_{\bmu}$, we consider the $\gl_{P+Q}$-module $\tF{0}{P,Q}$ as a module for its maximal compact subalgebra $\mathfrak k'=\gl_P \oplus \gl_Q$, and use \eqref{eq:kmodule}.  Since $\la=0$, equation \eqref{eq:kmodule} implies that
\begin{equation}
\label{invariant}
\tF{0}{P,Q} \cong \C\left[M_{P,Q}\right]
\end{equation}
as a $\mathfrak k'$-module, and therefore as a representation of $K'=\GL_P\times \GL_Q$, acting on the right-hand side as in \eqref{eq:kaction}.

We decompose this module by restricting again. Let $K=\GL_{p_1}\times \cdots\times\GL_{p_r}\times\GL_{q_1}\times\cdots\times\GL_{q_r}\hookrightarrow K'$ block-diagonally.
It follows from the definition of the action in \eqref{eq:kaction} and block matrix multiplication that
\begin{equation}
\label{eq:modulebreakdown}
\C\left[M_{P,Q}\right]\cong \bigotimes_{i,\, j=1}^{r}\C[M_{p_{i}, q_{j}}]
\end{equation}
as a $K$-representation,
with $\GL_{p_i}\times\GL_{q_j}$ acting on $\C\left[M_{p_i, q_j}\right]$ via \eqref{eq:kaction}.  Using \eqref{eq:kmodule} to decompose $\tF{\mu_1}{p_1,q_1} \otimes \cdots \otimes \tF{\mu_r}{p_r,q_r}$ as a $K$-representation, it follows from \eqref{eq:modulebreakdown} and \eqref{eq:reciprocity} that
\begin{equation}
\label{eq:bigdecomp}
\bigotimes_{i,j=1}^{r} 
\C[M_{p_i, q_j}]
\cong 
\bigoplus_{\bmu} m_{\bmu} \left(\bigotimes_{k=1}^r \C[M_{p_k,q_k}] \otimes \gF{\mu^+_k}{p_k} \otimes \gF{\mu^-_k}{q_k}\right)_{\textstyle .}
\end{equation}

Now let $\mathfrak t\subset \gl_{P+Q}$ be the standard
Cartan subalgebra of diagonal matrices, and let $T\subset \GL_{P+Q}$ be the
corresponding Cartan subgroup. Taking formal $T$-characters in \eqref{eq:bigdecomp}, we obtain
\begin{equation}
\label{eq:Tchar}
\prod_{i,j=1}^r \ch\left(\C[M_{p_i,q_j}]\right) = \sum_{\bmu} m_{\bmu}\prod_{k=1}^r \ch \left(\C[M_{p_k,q_k}]\right)\cdot \ch \left(\gF{\mu^+_k}{p_k}\right) \cdot \ch\left( \gF{\mu^-_k}{q_k}\right).
\end{equation}
Canceling the factors $\ch\left(\C[M_{p_i, q_i}]\right)$ for $i=1,\dots , r$ from both sides of \eqref{eq:Tchar}, we have
$$
\prod_{i\neq j}\ch\left( \C[M_{p_i,q_j}]\right)=\sum_{\bmu}m_{\bmu} \prod_k \ch \left( \gF{\mu^+_k}{p_k}\right)\cdot \ch \left( \gF{\mu^-_k}{q_k}\right)_{\textstyle .}
$$
Since the action of $K$ is locally finite, this equality of characters implies the following isomorphism of $K$-representations:
\begin{equation}
\label{eq:secondbigiso}
\bigotimes_{i\neq j} \C[M_{p_i,q_j}] \cong \bigoplus_{\bmu} m_{\bmu}\bigotimes_k \gF{\mu^+_k}{p_k} \otimes \gF{\mu^-_k}{q_{k.}}
\end{equation}
We will now decompose the left-hand side of \eqref{eq:secondbigiso} as a $K$-representation.  It follows from \eqref{howepq} that
as a $\GL_{p_i}\times \GL_{q_j}$-representation, under the action in \eqref{eq:kaction}, we have
\begin{equation}
\label{eq:multfree}
\C[M_{p_i, q_j}]\cong\bigoplus_{\la_{ij}} \gF{\la_{ij}}{p_i}\otimes \gF{\la_{ij}}{q_j}
\end{equation}
where the direct sum is over all partitions $\la_{ij}$ such that $\ell(\la_{ij})\leq \min\{p_i, q_j\}$.
Using \eqref{eq:multfree}, we see that \eqref{eq:secondbigiso} becomes
\begin{equation}
\label{eq:aftermultfree}
\bigotimes_{i\neq j} \bigoplus_{\la_{ij}}\gF{\la_{ij}}{p_i}\otimes \gF{\la_{ij}}{q_j} \cong \bigoplus_{\bmu} m_{\bmu}\bigotimes_k \gF{\mu^+_k}{p_k}\otimes\gF{\mu^-_k}{q_k.}
\end{equation}
Interchanging the tensor product and direct sum operations, we can rewrite the left-hand side of \eqref{eq:aftermultfree} as the $K$-representation
\begin{equation}
\label{hugetensor}
\bigoplus_\Lambda \Big(\gF{\la_{1,2}}{p_1} \otimes \cdots \otimes \gF{\la_{1,r}}{p_1} \otimes \gF{\la_{2,1}}{q_1} \otimes \cdots \otimes \gF{\la_{r,1}}{q_1} \otimes \cdots \otimes 
\gF{\la_{r,1}}{p_r} \otimes \cdots \otimes \gF{\la_{r,r-1}}{p_r} \otimes \gF{\la_{1,r}}{q_r} \otimes \cdots \otimes \gF{\la_{r-1,r}}{q_r}\Big),
\end{equation}
where the direct sum is over all hollow $r \times r$ matrices of partitions $\Lambda = (\la_{ij})$ with $\ell(\la_{ij})\leq \min\{p_i, q_j\}$, since $i\neq j$ in \eqref{eq:aftermultfree}.  Then \eqref{eq:aftermultfree} implies that $m_{\bmu}$ is the multiplicity of the irreducible $K$-representation $\gF{\mu^+_1}{p_1} \otimes \gF{\mu^-_1}{q_1} \otimes \cdots \otimes \gF{\mu^+_r}{p_r} \otimes \gF{\mu^-_r}{q_r}$
in the $K$-representation \eqref{hugetensor}.  For a fixed summand in \eqref{hugetensor}, this multiplicity is 0 unless $\Lambda$ is an LR-contingency table with margins $\bmu$, in which case $m_{\bmu}=\| \Lambda \|_{\bmu}$.  Therefore $m_{\bmu}=\sum_\Lambda \| \Lambda \|_{\bmu}$.  The result now follows from \eqref{mmu}.
\end{proof}

\section{Special cases and corollaries}
\label{s:special}

Our result lends itself to a wide range of special cases which reduce to counting classical contingency tables.

\begin{cor}
\label{c:onepart}
Let $\bmu=(\mu_1,\dots,\mu_r)$ where each $\mu_i=(a_i,0,0,\dots,0,0,-b_i)$ for $a_i,b_i \in \mathbb Z_{\geq 0}$.  If $n \geq 2r$, then $\LRC{0}{\bmu}$ is equal to the number of hollow contingency tables with margins $(a_1,\dots,a_r)$ and $(b_1,\dots,b_r)$.
\end{cor}

\begin{proof} 
We have $\mu_i=((a_i),(b_i))$ in the notation from \eqref{combine}; therefore $\mu^+_i=(a_i)$ and $\mu^-_j=(b_j)$.  The Pieri rule (see \cite{fulton}, p.\ 24) tells us that 
\begin{equation*}
\LRC{(a_i)}{\blarow{i}}=
\begin{cases}
    1 & \text{if }\ell\left(\la_{ij}\right)\leq1 \text{ for all } j=1,\dots,r \text{, and } \displaystyle\sum_{j=1}^r |\la_{ij}|=a_i,\\
    0 & \text{otherwise}.
\end{cases}
\end{equation*}
Likewise,
\begin{equation*}
\LRC{(b_j)}{\blacol{j}}=
\begin{cases}
    1 & \text{if }\ell\left(\la_{ij}\right)\leq1 \text{ for all } i=1,\dots,r \text{, and } \displaystyle\sum_{i=1}^r |\la_{ij}|=b_j,\\
    0 & \text{otherwise}.
\end{cases}
\end{equation*}
Therefore, for any LR-contingency table $\Lambda$ with margins $\bmu$, we have $\| \Lambda \|_{\bmu}=1$; moreover, each entry is a one-part partition, which is naturally identified with a nonnegative integer.  Hence any such $\Lambda$ is actually a classical contingency table with margins $(a_1,\ldots,a_r)$ and $(b_1,\ldots,b_r)$ and zeros down the diagonal.  Theorem \ref{thm:bigthm} now implies that $\LRC{0}{\bmu}$ equals the number of such contingency tables.
\end{proof}

As in the previous corollary, all of the following examples are applications of Theorem \ref{thm:bigthm} and the Pieri rule.

\begin{exam}
As mentioned in Section \ref{s:prelim},  $\F[(1)]{(1)}$ is the nontrivial component of the adjoint representation of $G=\GL_n$.  We now compute $\dim\left(\textstyle\bigotimes^r \F[(1)]{(1)}\right)^G$ when $n \geq 2r$.

It follows from Corollary \ref{c:onepart} that this dimension is equal to the number of hollow $r\times r$ matrices with nonnegative integer entries whose row and column sums are all $1$.  We recognize such matrices as $r\times r$ permutation matrices with no fixed points, which correspond to the \emph{derangements} of an $r$-element set.  The number of derangements is often denoted by $!r$, and so 
\begin{equation*}
 \dim\left(\textstyle\bigotimes^r \F[(1)]{(1)}\right)^G=\:!r.
\end{equation*}
\end{exam}

\begin{cor}
\label{c:realcontingency}
Let $G=\GL_n$, where $n \geq r+s$.  Then for $a_i,b_j\in \mathbb Z_{\geq 0}$, the dimension
\begin{equation*}
    \dim \hom_G\left(\bigotimes_{i=1}^r \F{(a_i)},\:\bigotimes_{j=1}^s \F{(b_j)}\right)
\end{equation*}
is equal to the number of $r \times s$ contingency tables with margins $(a_1,\dots,a_r)$ and $(b_1,\dots,b_s)$.
\end{cor}

\begin{proof}
We have
\begin{align*}
    \dim \hom_G\left(\bigotimes_{i=1}^r \F{(a_i)},\bigotimes_{j=1}^s \F{(b_j)}\right)&=\dim\left[\left(\bigotimes_{i=1}^r \F{(a_i)}\right)^*\otimes \left(\bigotimes_{j=1}^s \F{(b_j)}\right)\right]^G\\
    &=\dim \left[ \left( \bigotimes_{i=1}^r \F[0]{(a_i)} \right) \otimes \left(\bigotimes_{j=1}^s \F{(b_j)}\right)\right]^G\\[2ex]
    &= \LRC{0}{\bmu,}
\end{align*}
where
\begin{equation*}
\bmu^+=\Big(\underbrace{0,\dots,0}_r,(b_1),\dots,(b_s)\Big) \quad \text{and} \quad \bmu^-=\Big((a_1),\dots,(a_r),\underbrace{0,\dots,0}_s\Big).
\end{equation*}
Since all partitions in $\bmu^+$ and $\bmu^-$ have length at most 1, we can apply Corollary \ref{c:onepart}.  Therefore, $\LRC{0}{\bmu}$ is equal to the number of $(r+s) \times (r+s)$ matrices with block form
\begin{equation*}
    \begin{bmatrix}
    0&0\\
    C&0
    \end{bmatrix}_{\textstyle ,}
\end{equation*}
where $C$ is an $s \times r$ contingency table (not necessarily hollow) with margins $(b_1,\dots,b_s)$ and $(a_1,\dots,a_r)$.
\end{proof}

The remaining examples in this section are motivated by Schur-Weyl duality; by counting contingency tables, we prove statements in $\GL_n$ whose analogues in the symmetric group are well-known classical results.

\begin{cor}
\label{c:schur}
Let $G=\GL_n$, with $n \geq 2r$. Then $\dim \normalfont\text{End}_G\Big( \bigotimes^r\C^n\Big)=r!$.
\end{cor}
\begin{proof}
Recalling that $\F{(1)}\cong \C^n$, we have
\begin{equation*}
    \textstyle\dim \text{End}_G\Big( \bigotimes^r\C^n\Big)=\dim \hom_G\Big( \bigotimes^r\F{(1)},\:\bigotimes^r \F{(1)}\Big).
\end{equation*}
Hence we can apply Corollary \ref{c:realcontingency} where $a_i=b_i=1$ for all $i=1,\ldots,r$.  It follows that the desired dimension equals the number of $r\times r$ contingency tables whose rows and columns all sum to 1.  But these are precisely the $r\times r$ permutation matrices, and so
\begin{equation*}
    \textstyle
    \dim \text{End}_G\Big( \bigotimes^r\C^n\Big) = |S_r|=r!.
\end{equation*}
\end{proof}

\begin{exam}
Let $\la$ be a partition whose Young diagram has $r$ rows and $s$ columns. For this example, we will use the more standard subscript notation $\la_i$ to denote the $i^\text{th}$ part of $\la$, and we will use exponents to designate repeated parts in a partition, so that $(1^k)=(\underbrace{1,\ldots,1}_k)$.   Assuming $n\geq r+|\la|$, we claim
\begin{equation}
\label{symmetrizer}
    \left[ \bigotimes_{j=1}^s\textstyle\bigwedge^{\la'_j}(\C^n) : \displaystyle\bigotimes_{i=1}^r S^{\la_i}(\C^n) \right]=1,
\end{equation}
as representations of $\GL_n$.  (Some readers may recall the corresponding result for the symmetric group from Schur--Weyl duality.) Indeed, rewriting the left-hand side of \eqref{symmetrizer} as representations of $G=\GL_n$, we have
\begin{align*}
    \left[ \bigotimes_{j=1}^s \F{(1^{\la'_j})} : \bigotimes_{i=1}^r \F{(\la_i)} \right]&=\dim \left[ \left(\bigotimes_{j=1}^s \F{(1^{\la'_j})}\right)^*\otimes \left( \bigotimes_{i=1}^r \F{(\la_i)} \right)\right]^G\\
    &=\dim \left[ \left(\bigotimes_{j=1}^s \F{\left(0,\:(1^{\la'_j})\right)}\right)\otimes \left( \bigotimes_{i=1}^r \F{(\la_i)} \right)\right]^G\\[2ex]
    &= \LRC{0}{\bmu,}\\
\end{align*}
where $\bmu^+=(0^s,(\la_1),\ldots,(\la_r))$ and $\bmu^-=((1^{\la'_1}),\ldots, (1^{\la'_s}), 0^r)$.  Any LR-contingency table with margins $\bmu$ is an $(r+s)\times (r+s)$ matrix with block form $\Lambda = \left[\begin{smallmatrix} 0&0\\C&0 \end{smallmatrix}\right]$.  By the Pieri rule, $C$ must be a contingency table whose entries are either $0$ or $1$, with margins $(\la_1,\ldots,\la_r)$ and $(\la'_1,\ldots,\la'_s)$.  But there is exactly one such table $C$, which we construct by filling with 1's in the shape of the Young diagram of $\la$, and 0's elsewhere.  By the Pieri rule, $\| \Lambda \|_{\bmu}=1$, and so by Theorem \ref{thm:bigthm}, we have \eqref{symmetrizer}.
\end{exam}

\begin{exam}  We consider the $r$-fold tensor product of $\F{\la}$ where $\la$ is a partition whose Young diagram is a \textit{hook}, meaning every row except the first has at most one box.  We include this example because hooks are relevant when examining the exterior algebra of defining representations of the symmetric group (see \cite{gw}, Section 9.2.4 exercises). 

Let $r\geq 1$ and $s\geq r^2 + r$, and assume $n >r+s$.  We claim
\begin{equation}
\label{hooks}
    \left[ \F{((r+1)^r,r^s)}:\textstyle \bigotimes^r \F{(r+1,1^s)}\right]=1,
\end{equation}
where the exponents designate repeated parts as in the previous example.  Note that the partition $(r+1,1^s)$ is a hook.  The Young diagram of $((r+1)^r,r^s)$ is an $(r+1)\times(r+s)$ rectangle, with a vertical strip of length $s$ removed from the lower-right.  As a concrete example, take $r=2$ and $s=7$; then letting a Young diagram $\la$ denote the corresponding representation $\F{\la}$, our claim is
\begin{equation*}
\tiny\Yvcentermath1\yng(3,3,2,2,2,2,2,2,2) \hookrightarrow \quad \Yvcentermath1\yng(3,1,1,1,1,1,1,1) \otimes \quad\Yvcentermath1\yng(3,1,1,1,1,1,1,1) \normalsize\text{ with multiplicity 1.}
\end{equation*}

To prove the claim, let $G=\GL_n$, with $n=r+s+1$.  Note that our condition on $s$ implies  $n\geq(r+1)^2$.  By padding out with 0's, the claim \eqref{hooks} can be written $\left[ \F{((r+1)^r,r^s,0)}:\bigotimes^r \F{(r+1,1^s,0^r)}\right]=1$.  Let $V=\F{((r+1)^r,r^s,0)}$ and $W=\F{(r+1,1^s,0^r)}$.  Then 
\begin{align*}
    [V:W^{\otimes r}] &=\dim \left(V^*\otimes W^{\otimes r}\right)^G\\
    &= \dim \left[(V\otimes \det^{-r})^*\otimes (W \otimes \det^{-1})^{\otimes r}\right]^G\\
    &= \dim \left[ \left( \F{(1^r,0^s,-r)} \right)^* \otimes \textstyle \bigotimes^r \F{(r,0^s,-1^r)}\right]^G\\
    &=\dim\left[\textstyle \bigotimes^{r+1} \F{(r,0^s,-1^r)} \right]^G\\
    &= \LRC{0}{\bmu,}
\end{align*}
where $\bmu^+=((r),\ldots,(r))$ and $\bmu^-=((1^r),\ldots,(1^r))$.  Since $n\geq (r+1)^2$, we are in the stable range.  By Theorem \ref{thm:bigthm} and the Pieri rule, $\LRC{0}{\bmu}$ is the number of hollow $(r+1) \times (r+1)$ contingency tables with entries from $\{0,1\}$, where both margins are $(r,\ldots,r)$.  But there is exactly one such contingency table, namely the matrix with 0's on the diagonal and 1's elsewhere.  Since the result holds for $n=r+s+1$, it holds for all $n > r+s$ by stability.
\end{exam}

\section{Result for $\O_n$ and $\Sp_{2n}$}
\label{s:ortho}

Our method above, which began by letting $G_0$ be the real compact group $\U(n)$, can be easily adapted for $G_0=\O(n)$.  In this case, we compute tensor product multiplicities for the complex orthogonal group $G=\O_n:=\O(n,\C)$ using \textit{symmetric} hollow LR-contingency tables.  Further, since the tensor product multiplicities for the symplectic group coincide with those of the orthogonal group within the stable range, our result holds for $\Sp_{2n}:=\Sp(2n,\C)$. Thus, taking Theorem \ref{thm:bigthm} together with Theorem \ref{bigthmortho} below, we can compute tensor product multiplicities for any complex classical group using LR-contingency tables. 

For $G=\O_n$, the method and proof follow the $\GL_n$ case above \emph{mutatis mutandis}, and so we simply list the changes in the table below.  The details behind each entry can be either found in or inferred from the proofs in \cite{htw}, whose notation we follow here; note that $SM_p$ denotes the space of symmetric $p \times p$ complex matrices.
\begin{center}
{\renewcommand{\arraystretch}{1.5}
\begin{longtable}{|c||c|c|}\hline
    &$G=\GL_n$ & $G=\O_n$\\ \hline \hline
    Compact pair $G_0$---$\mathfrak g'_0$ &  $\U(n)$---$\mathfrak{u}(P,Q)$ & $\O(n,\mathbb R)$---$\sp(2P,\mathbb R)$\\ \hline
    $G$---$\mathfrak g'$ & $\GL_n$---$\gl_{P+Q}$ & $\O_n$---$\sp_{2P}$\\ \hline
    Rational $G$-irreps & $\F{\mu_i}$ & $E^{\mu_i}_n$\\ \hline
Conditions on $\mu_i$ & weakly decreasing $n$-tuple & \makecell{(nonnegative) partition, with\\$\mu_i'(1)+\mu_i'(2)\leq n$}\\ \hline
    $K'$ & $\GL_P \times \GL_Q$ & $\GL_P$\\ \hline
    $K$ & $\GL_{p_1} \times \GL_{q_1} \times \cdots \times \GL_{p_r}\times \GL_{q_r}$ & $\GL_{p_1} \times \cdots \times \GL_{p_r}$\\ \hline
    Seesaw pair & $\begin{array}{ccc}
    (\GL_n)^r & \leftrightarrow & \bigoplus_i \gl_{p_i+q_i}\\
    \uparrow & & \downarrow \\
    \Delta(\GL_n) & \leftrightarrow & \gl_{P+Q}
\end{array}$ & 
$\begin{array}{ccc}
    (\O_n)^r & \leftrightarrow & \bigoplus_i \sp_{2p_i}\\
    \uparrow & & \downarrow \\
    \Delta(\O_n) & \leftrightarrow & \sp_{2P}
\end{array}$ \\\hline
Choice of $p_i$, $q_i$ & $p_i = \ell(\mu^+_i)$, \quad $q_i=\ell(\mu^-_i)$ & $p_i = \ell(\mu_i)$\\ \hline
Stable range & $n\geq P+Q$ & $n \geq 2P$\\ \hline
$V$ & $M_{n,P} \oplus M_{n,Q}$ & $M_{n,P}$\\ \hline
Inf.-dim.\ $\mathfrak g'$-modules & $\tF{\mu_i}{p_i,q_i}$ & $\tE{\mu_i}{2p_i}$ \\ \hline
$K'$-decomposition & $\cong \C[M_{p_i,q_i}]\otimes \gF{\mu^+_i}{p_i}\otimes \gF{\mu^-_i}{q_i}$& $\cong \C[SM_{p_i}]\otimes \gF{\mu_i}{p_i}$ \\ \hline
$m_{\bmu}$ &$\left[\bigotimes_i \tF{\mu_i}{p_i,q_i}:\tF{0}{P,Q}\right]$ & $\left[\bigotimes_i \tE{\mu_i}{2p_i}:\tE{0}{2P}\right]$\\ \hline
$\C[V]^G$ & $\tF{0}{P,Q}\cong \C[M_{P,Q}]$& $\tE{0}{2P}\cong \C[SM_P]$\\ \hline
$K$-decomp.\ of $\C[V]^G$ & $\displaystyle \bigotimes_{i,j} \C[ M_{p_i,q_j}]$&$ \displaystyle\bigotimes_i \C[SM_{p_i}] \otimes \bigotimes_{i<j} \C[M_{p_i,p_j}]$\\ \hline
\end{longtable}}
\end{center}
Using the last line of the table, we begin with the analogue of \eqref{eq:bigdecomp}, and working with formal $T$-characters as in the proof of Theorem \ref{thm:bigthm}, we obtain the analogue of \eqref{eq:secondbigiso}:
\begin{align*}
    \bigotimes_i \C[SM_{p_i}] \otimes \bigotimes_{i<j} \C[M_{p_i,p_j}] &\cong \bigoplus_{\bmu}m_{\bmu} \bigotimes_k \C[SM_{p_k}]\otimes \gF{\mu_k}{p_k} \\
    \prod_i \ch \left(\C[SM_{p_i}]\right) \prod_{i<j}\ch \left(\C[M_{p_i,p_j}]\right) &= \sum_{\bmu} m_{\bmu} \prod_k \ch \left(\C[SM_{p_k}]\right)\cdot \ch \left(\gF{\mu_k}{p_k}\right) \\
    \prod_{i<j}\ch \left(\C[M_{p_i,p_j}]\right) &= \sum_{\bmu} m_{\bmu} \prod_k \ch\left(\gF{\mu_k}{p_k}\right) \\
    \bigotimes_{i<j}\C[M_{p_i,p_j}] &\cong \bigoplus_{\bmu} m_{\bmu} \bigotimes_k \gF{\mu_k}{p_{k.}}
\end{align*}
Applying \eqref{howepq} to each tensor factor in the left-hand side, we have
\begin{equation}
\label{eq:orthiso}
    \bigotimes_{i<j}\bigoplus_{\la_{ij}} \gF{\la_{ij}}{p_i}\otimes \gF{\la_{ij}}{p_j} \cong \bigoplus_{\bmu} m_{\bmu} \bigotimes_k \gF{\mu_k}{p_k}
\end{equation}
for all partitions $\la_{ij}$ such that $\ell(\la_{ij})\leq \min\{p_i,p_j\}$. The condition $i<j$ in \eqref{eq:orthiso} requires us to define a different kind of contingency table than before, in which we take the generalized LRC's across rows only (or equivalently, columns only).

\begin{dfn}
Let $\Lambda=(\la_{ij})$ be a symmetric $r \times r$ matrix of partitions.  Then
\begin{equation*}
    \llbracket \Lambda \rrbracket_{\bmu} := \prod_{i=1}^r \LRC{\mu_i}{\blarow{i}}
\end{equation*}
\end{dfn}

\begin{dfn}
Let $\bmu=(\mu_1,\ldots,\mu_r)$ be a vector of partitions with $\ell(\mu_i)\leq n$.  Then an \emph{$LR_{Sym}$-contingency table with margins }$\bmu$ is a symmetric $r \times r$ matrix of partitions $\Lambda = (\la_{ij})$ such that $\llbracket \Lambda \rrbracket_{\bmu}\neq 0$.
\end{dfn}

Now interchanging the tensor product and direct sum, we can rewrite the left-hand side of \eqref{eq:orthiso} as 
\begin{equation}
    \label{hugetensorortho}
    \bigoplus_\Lambda \Big( \gF{\la_{1,2}}{p_1} \otimes \cdots \otimes \gF{\la_{1,r}}{p_1} \otimes \cdots \otimes \gF{\la_{1,r}}{p_r} \otimes \cdots \otimes \gF{\la_{r-1,r}}{p_r}\Big),
\end{equation}
where the direct sum is over all hollow \textit{symmetric} matrices of partitions $\Lambda$ with $\ell(\la_{ij})\leq \min\{p_i,p_j\}$.  Since any $K$-representation $\bigotimes_{i=1}^r \gF{\mu_i}{p_i}$ occurring in a fixed summand of \eqref{hugetensorortho} has multiplicity $\llbracket \Lambda \rrbracket_{\bmu}$, we obtain the following result for the orthogonal and symplectic groups.

\begin{thm}
\label{bigthmortho}
Let $\bmu =(\mu_1,\dots,\mu_r)$ be a
vector of partitions. Let $E^{\mu_i}_n$ denote the irreducible representation of $\O_n$ labeled by the partition $\mu_i$, and let $V^{\mu_i}_{2n}$ denote the irreducible reprentation of $\Sp_{2n}$ with highest weight $\mu_i$.  If $n\geq 2 \sum_{i=1}^r \ell(\mu_i)$, then
\begin{equation*}   
\dim \left( E^{\mu_1}_n \otimes \cdots \otimes E^{\mu_r}_n \right)^{\O_n} =
\dim\left( V^{\mu_1}_{2n} \otimes \cdots \otimes V^{\mu_r}_{2n}\right)^{\Sp_{2n}} =
\sum_\Lambda \llbracket \Lambda \rrbracket_{\bmu,}
\end{equation*}
where $\Lambda$ ranges over all hollow $ LR_ {Sym}$-contingency tables with margins $\bmu$.
\end{thm}

\begin{exam}
If $n\geq 2r$, then we claim
\begin{equation*}
    \dim\left( \textstyle \bigotimes^r V^{(1)}_{2n} \right)^{\Sp_{2n}} =
    \dim \left( \textstyle \bigotimes^r E^{(1)}_n \right)^{\O_n}=
    \begin{cases}
    (r-1)!!& \text{if } r \text{ is even},\\
    0 & \text{if }r\text{ is odd}.
    \end{cases}
\end{equation*}
Indeed, setting $\bmu=((1),\ldots,(1))$, it is clear by the Pieri rule that any ${\rm LR}_{\rm Sym}$-contingency table $\Lambda$ with margins $\bmu$ has exactly one $(1)$ in each row and 0's elsewhere, which implies $\llbracket \Lambda \rrbracket_{\bmu}=1$.
Each such $\Lambda$ is a hollow symmetric $r\times r$ permutation matrix, and the number of these matrices is precisely the number of fixed-point-free involutions on $r$ letters.  No such involution exists when $r$ is odd; if $r$ is even, then the number of such involutions is known to be $(r-1)!!=(r-1)(r-3)(r-5)\cdots 5 \cdot 3 \cdot 1$.
\end{exam}

\begin{exam}
Recall that in the standard coordinates, $2\epsilon_1$ is the highest root for the Lie algebra $\mathfrak{sp}_{2n}$.  Hence the highest weight of the adjoint representation of $\Sp_{2n}$ is labeled by the one-part partition $(2)$.  Let $V:=V^{(2)}_{2n}$ denote this adjoint representation. We now describe the dimension of the space of invariants of $V^{\otimes r}$, assuming $n \geq 2r$.

By Theorem \ref{bigthmortho} we have $\dim (V^{\otimes r})^{\Sp_{2n}}=\sum_\Lambda \llbracket \Lambda \rrbracket_{\bmu}$, where $\Lambda$ is a hollow ${\rm LR}_{\rm Sym}$-contingency table with margins $\bmu=((2)^r)$.  By the Pieri Rule, any such $\Lambda$ is a hollow contingency table (in the standard sense) whose rows all sum to 2; the symmetric condition then forces the column sums to be 2 as well.  For each such $\Lambda$, the Pieri rule gives $\llbracket \Lambda \rrbracket_{\bmu}=1$.  Therefore, $\dim (V^{\otimes r})^{\Sp_{2n}}$ coincides with the number of hollow symmetric $r\times r$ contingency tables with both margins $(2,\ldots,2)$.  (These are two of the many characterizations given for the integer sequence A002137 in the OEIS.)

This same result holds if we set $V=E^{(1,1)}_n$, the irreducible representation of $\O_n$ which restricts to the adjoint representation of ${\rm SO}_n$.
\end{exam}

\bibliographystyle{amsplain}
\bibliography{references}

\end{document}